\documentclass[12pt]{amsart}
\usepackage[all]{xy}
\usepackage{amssymb}
\usepackage{amsthm}
\usepackage{hyperref}
\hypersetup{colorlinks=true,linkcolor=blue,citecolor=magenta}
\usepackage{amsmath}
\usepackage{amscd,enumitem}
\usepackage{verbatim}
\usepackage{eurosym}
\usepackage{graphicx}
\usepackage{float}
\usepackage{color}
\usepackage{dcolumn}
\usepackage[mathscr]{eucal}
\usepackage[all]{xy}
\usepackage{hyperref}
\usepackage{bbm}
\usepackage[textheight=8.65in, textwidth=6.8in]{geometry}
\usepackage{multirow}
\usepackage{caption}
\newtheorem*{thm*}{Theorem}
\newtheorem*{conj*}{Conjecture}

\newtheorem{thm}{Theorem}[section]

\newtheorem{cor}[thm]{Corollary}

\newtheorem{prop}[thm]{Proposition}

\newcommand{\Z}{\mathbb{Z}}
\newcommand{\Q}{\mathbb{Q}}

\newcommand{\N}{\mathbb{N}}
\newcommand{\SL}{\operatorname{SL}}

\numberwithin{equation}{section}

\begin{document}
\title{prime values of Ramanujan's tau function}
\author{Boyuan Xiong}
\address{Department of Mathematics, Indiana University, Bloomington, IN 47405, USA}
\email{boyxiong@iu.edu}

\begin{abstract}
We study the prime values of Ramanujan's tau function $\tau(n)$. Lehmer \cite{Lehmer2} found that $n=251^2=63001$ is the smallest $n$ such that $\tau(n)$ is prime: $$\tau(251^2)=-80561663527802406257321747.$$ We prove that in most arithmetic progressions (mod 23), the prime values $\tau$ belonging to the progression form a thin set. As a consequence, there exists a set of primes of Dirichlet density $\frac{9}{11}$ which are not values of $\tau$.
\end{abstract}
\maketitle

\section{Introduction}
In 1916, Ramanujan \cite{Ramanujan1916} defined the $\tau$-function  $\tau(n)$ as following: $$\Delta(q)=\sum_{n=1}^\infty\tau(n)q^n=q\prod_{n=1}^\infty(1-q^n)^{24}=q-24q^2+252q^3-1472q^4+4830q^5-\cdots,$$\\which is the normalized weight 12 cusp form of $\SL_2(\Z)$. Lehmer's celebrated conjecture \cite{Lehmer}, which states that $\tau(n)$ is never 0, still remains open. However, Lehmer also showed that if $\tau(n)=0$, then there exists a prime $p$ for which $\tau(p)=0$. Serre~\cite{serre1981} proved that the natural density of primes $p$ for which $\tau(p)=0$ is $0$. This provides strong evidence in support of the conjecture. It is worth emphasizing that our results are of a different nature from those 
connected to Lehmer's conjecture.  The classical problem, studied by Serre~\cite{serre1981} 
and others, concerns the density of primes $p$ for which $\tau(p)=0$.  
By contrast, the present paper investigates the density of primes $p$ which 
occur as values of the $\tau$-function, that is, the distribution of 
solutions to $\tau(n)=p$ with $p$ prime.  In this way, our results complement 
the extensive literature on Lehmer’s conjecture by considering the density of 
prime values in the image of~$\tau$ rather than the zero set of $\tau$ at primes.\\

There has been recent interest in the set of omitted values of the $\tau$-function, that is integers which do not lie in the image of $\tau$.  For an odd $\alpha$, Murty, Murty and Shorey \cite{MMS} proved that $\tau(n) = \alpha$ for at most finitely many $n$. More precisely, they proved that there is a effective computable constant $c$ such that if $\tau(n)$ is odd, then $$|\tau(n)| > (\log n)^c.$$\\
In recent years, a number of papers have focused on finding explicit values $\alpha$ such that $\tau(n)=\alpha$ has no solutions. For odd values $\alpha$, Balakrishnan, Craig, Ono and Tsai \cite{BCOT}  provide a strategy for proving that certain odd integers are not $\tau$-values: \\
\begin{prop}\label{1.1}
    Suppose that $l$ is an odd prime such that $l \nmid \tau(l)$. If $\tau(n) = \pm l^m$, with $m \in \Z^+$, then $n = p^{d-1}$, where $p$ and $d \text{ }|\text{ } l(l^2-1)$ are odd primes. Furthermore, $\tau(n) = \pm l^m$ for at most finitely many $n$.
\end{prop}

Using this proposition, Balakrishnan, Craig, Ono and Tsai \cite{BCOT} proved that $\tau(n)$ can never belong to $$\{\pm1,\pm3,\pm5,\pm7,\pm13,\pm17,-19,\pm23,\pm37,\pm691\}$$
for $n\geq2$. After that, many authors have used Proposition \ref{1.1} to find additional omitted values of $\tau$. Following work of Amir and Hatziiliou \cite{AH1}, Amir and Hong \cite{AH2}, Dembner and Jain \cite{DJ}, Hanada and Madhukara \cite{HM} proved that $\tau(n) \neq \alpha$ for odd $\alpha$ with $|\alpha| < 100$. Furthermore, Bennett, Gherga, Patel and Siksek \cite{BGPS} proved that $\tau(n) \neq \pm l^m$ for all odd primes $3 \leq l < 100$ and all positive integer $m$.\\

Regarding even values, Balakrishnan, Ono and Tsai \cite{BOT} proved that$$\tau(n)\notin\{2l^j\colon (l,j)\in N^+\}\cup\{-2l^j\colon (l,j)\in N^-\}$$where $N^\pm=\{(l,j)\colon 1\leq j\not\equiv r\textup{ (mod }t)\textup{ for all }(l,r,t)\in S^{\pm}\}$ and \begin{equation}
S^{+}:= \left\{ \begin{matrix} (3, 0, 44), (5,0,22), (7,0,44), (7,19,44), (11,0,22), (13, 0, 44), (17, 0, 44), \\ (19, 0, 22),  (23,0,4), (29,0,22), 
(31,0,22), (37,0,44), (37,35,44), (41,0,22),\\ (43,0,44), (43,37,44), (47,0,4), (53,0,44), (59, 0, 22), (61,0,22), (67, 0, 44), \\ (67,43,44),
(71,0,22), (73,0,44), (79,0,22), (83,0,44), (89,0,22), (97,0,44)
\end{matrix}
\right \}
\end{equation}
\begin{equation}
S^{-}:=\{(3,15, 44), (5, 11, 22), (17, 33, 44), (59, 3, 22), (83,11, 44), (89, 11, 22)\}.
\end{equation}\\
The goal of all previous work on this problem has been to identify specific integers which are not $\tau$-values. The growth rate of the sequence of integers for which this has been done is exponential. We take a different slant, trying to show that there are many omitted values, even if we cannot identify them. Let $T=\{x\in\Z\colon x\neq\tau(n)\textup{ for all }n\in\N\}$ be the set of integers that are not values of $\tau(n)$. Also, for any set $X\subseteq\Z$, let $\pi_X$ be the usual counting function of $X$, that is,  $$\pi_X(N)= \# ( X\cap[-N,N]).$$
In this paper, without finding any explicit value that can be omitted by $\tau$, we can improve the result from $\pi_T(N) \gg \log N$ to $\pi_T(N)\gg\frac{N}{\log N}$ by proving the following theorem:
\begin{thm}\label{T}
    If $$b\in\{2,4,6,7,8,9,10,11,12,13,14,15,16,17,18,19,20,21\},$$ and $P_b=\{p \textup{ prime}\colon p \equiv b \textup{ (mod 23)}\}$. Then $\pi_{P_b\cap T}(N) \gg \frac{N}{\log N}$. 
\end{thm}

If $X$ is a set of primes (including negative primes), then we define the Dirichlet density of $X$ to be $$\lim_{s\to 1^+}\frac{\sum_{p\in X}\frac{1}{|p|^s}}{2\sum_{p>0\textup{ prime}}\frac{1}{p^s}}=\lim_{s\to 1^+}\frac{\sum_{p\in X}\frac{1}{|p|^s}}{2\log(\frac{1}{s-1})}.$$ 

Theorem~\ref{T} shows that the set of prime numbers which occur as $\tau$-values is extremely sparse among all primes. In terms of Dirichlet density, we may express this observation as follows:

\begin{cor}\label{cordiri}
    The set $P_b \cap T$ has Dirichlet density zero.
\end{cor}
Conversely, by Theorem \ref{T} and Prime Number Theorem for Arithmetic Progressions, we have the following:
\begin{cor}
    The set of primes which do not occur as $\tau$-values has Dirichlet density at least $\frac{9}{11}$.
\end{cor}

\section{Main idea}
Let $S$ be the set of values of $\tau(n)$, i.e. $$S=T^c=\{\tau(n)\colon n\in\mathbb{N}\}.$$

Let $P$ be the set of odd primes (including negative primes), i.e.$$P=\{\pm3,\pm5,\pm7,\pm11,\cdots\}.$$

Since $\Delta$ is a cusp form of weight 12, we have the following well-known proposition, where  Mordell \cite{Mordell} proved the Hecke multiplicativity of $\tau(n)$ and Deligne \cite{Deligne1} \cite{Deligne2} gave the bound for $|\tau(p)|$: 
\begin{thm}\label{1}
    The following are true:\\
    (1) If $\textup{gcd}(m, n)=1$, then $\tau(mn)=\tau(m)\tau(n)$.\\
    (2) If $p$ is prime and $m \geq 2$, then $$\tau(p^m)=\tau(p)\tau(p^{m-1})-p^{11}\tau(p^{m-2}).$$\\
    (3) If $p$ is prime and $\alpha_p$, $\beta_p$ are the roots of $x^2-\tau(p)x+p^{11}$, then $$\tau(p^k)=\frac{\alpha_p^{k+1}-\beta_p^{k+1}}{\alpha_p-\beta_p}.$$
Moreover, $|\tau(p)|\leq 2p^{\frac{11}{2}}$ and hence $\alpha_p$ and $\beta_p$ are complex conjugates.
\end{thm}
By Theorem \ref{1} (1) together with the fact that $\tau(n)\neq\pm1$ when $n\geq 2$, we know that if $\tau(n)=l$ is prime, then $n$ must be a prime power. Hence, we want to consider the following sets $$X_k=\{\tau(p^k)\mid p\text{ prime}\}.$$

Another well-known result which shows that $\Delta$ has a trivial residual mod 2 Galois representation is the following: \begin{prop}\label{mod2}
    $\tau(n)$ is odd if and only if $n$ is an odd square.
\end{prop}

It follows that 
$$S\cap P=\bigg(\bigcup_{k=1}^\infty X_k\bigg)\cap P=\bigg(\bigcup_{k=1}^\infty X_{2k}\bigg)\cap P.$$

That is, odd primes can only appear in $X_{2k}$ \footnote{Indeed, if $\tau(n)$ is an odd prime, then by Theorem \ref{1} (3) we can easily show that $n=p^{q-1}$ where both $p$ and $q$ are odd primes. This result was also proved by Lygeros and Rozier \cite{LR}. Furthermore, Proposition \ref{1.1} \cite{BCOT} provides a stronger result that $q$ also needs to be divisible by $l(l^2-1)$.}. For any odd prime $l\equiv b$ (mod 23) where $b \in \{2,4,6,7,8,9,10,11,12,13,14,15,16,17,18,19,20,21\}$, we will discuss the possibility that $l$ appears in $X_{2k}$. We divide these $X_{2k}$ into three parts. For $k$ small ($k=1,2$), we use mod 23 congruence relations of $\tau(n)$ to show that $l$ is not an element of $X_2$ and $X_4$; for $k$ large, we use Schinzel's result  \cite{Schinzel} to find a lower bound for $X_{2k}$ , which makes it impossible for $l$ to appear in $X_{2k}$; and for $k$ in the middle, we use Bombieri and van der Poorten's result \cite{Bombieri} to approximate $\#(P\cap X_{2k})$, and finally show that the density of primes in the values of $\tau(n)$ is really low.\\
\section{When $k$ is small}
Ramanujan \cite{Ramanujan1916} \cite{Ramanujan} \cite{BerndtOno1999} proved the following famous mod 23 congruence property of $\tau(n)$:
\begin{prop}\label{23}
If $p\neq 23$ is a positive prime, then \begin{equation}
    \tau(p)\equiv
    \begin{cases}
      0 \textup{ (mod 23)}  & \textup{if } (\frac{p}{23})=-1 \\
      2 \textup{ (mod 23)}& \textup{if } p=a^2+23b^2 \textup{ with }a,b\in\mathbb{Z}\\
      -1 \textup{ (mod 23)}& \textup{otherwise} 
    \end{cases}
\end{equation}
\end{prop}
By Theorem \ref{1} (2) and Proposition \ref{23}, we have 
\begin{cor}\label{C}
    If $p\neq 23$ is a prime, $k\in\mathbb{N}$, then 
    \begin{equation}
        \tau(p^k)\equiv
        \begin{cases}
            0 \textup{ or }1\textup{ (mod 23)}& \textup{if }(\frac{p}{23})=-1 \\
            k+1 \textup{ (mod 23)}& \textup{if } p=a^2+23b^2 \textup{ with }a, b\in\mathbb{Z}\\
            0\textup{ or} \pm 1 \textup{ (mod 23)}& \textup{otherwise }
        \end{cases}
    \end{equation}
\end{cor}

Due to this famous $\tau-$congruences, for $k=1,2$, $\tau(p^k)=l$ can only lie in the residue classes $1,3,5,22$ (mod 23).
Therefore if $l\equiv b$ (mod 23) where $$b\in\{2,4,6,7,8,9,10,11,12,13,14,15,16,17,18,19,20,21\},$$then $l$ can only appear in $X_{2k}$ where $k\geq 3$.

\section{When $k$ is large}

Suppose $\alpha$ and $\beta$ are the roots of a quadratic $x^2+Ax+B$ where $A,B$ are integers and the discriminant $A^2-4B<0$. If moreover $B\neq 1$, then gcd($\alpha,\beta$)=1 in the quadratic number field $\mathbb{Q}(\alpha)$ and $\alpha/\beta$ is not a root of unity. Hence $\alpha$ and $\beta$ satisfy all the conditions in Theorem 1 of Schinzel's paper \cite{Schinzel}. Then plugging in $\alpha,\beta$ into (8) of \cite{Schinzel} we get:
\begin{prop}\label{3}
    If $\alpha$ and $\beta$ are roots of the quadratic $x^2+Ax+B$ where $A$, $B$ are integers with $A^2-4B<0$ and $B\neq 1$, and if $\Phi_n(\alpha,\beta)$ is the cyclotomic polynomial of $\alpha$ and $\beta$, then \begin{equation}
    |\Phi_n(\alpha,\beta)|>n|\alpha|^{\frac{11}{13}\varphi(n)} \textup{ for all } n>c_1.
\end{equation}
Here $\varphi(n)$ is Euler's totient function, and $c_1>0$ is some absolute constant.
\end{prop}

\begin{cor}\label{4}
    Let $\alpha$, $\beta$ be algebraic integers satisfying all the conditions in Proposition \ref{3}. Let 
        $P_n(\alpha,\beta)=\frac{\alpha^n-\beta^n}{\alpha-\beta}$
    . Then \begin{equation}
        |P_n(\alpha,\beta)|>|\alpha|^{\frac{11}{26}n}\textup{ for all }n>2c_1^2,
    \end{equation}
    where $c_1$ is the absolute constant in Proposition \ref{3}.
\end{cor}

\begin{proof}
    For all $n>2c_1^2$:
    \begin{equation*}
        P_n(\alpha,\beta)=\prod_{d|n,d>1}\Phi_d(\alpha,\beta)=\prod_{d|n,1<d\leq c_1}\Phi_d(\alpha,\beta)\prod_{d|n,d>c_1}\Phi_d(\alpha,\beta).
    \end{equation*}
    Since $\Phi_n(\alpha,\beta)$ is in the ring of integers of the imaginary quadratic number field $\mathbb{Q}(\alpha)$, we have $|\Phi_d(\alpha,\beta)|\geq1$ for all $d>1$. Hence by Proposition \ref{3}, \begin{align*}
        |P_n(\alpha,\beta)|&\geq\Bigg|\prod_{d|n,d>c_1}\Phi_d(\alpha,\beta)\Bigg|>\prod_{d|n,d>c_1}d|\alpha|^{\frac{11}{13}\varphi(d)}>\prod_{d|n,d>c_1}|\alpha|^{\frac{11}{13}\varphi(d)}\\&=(|\alpha|^{\frac{11}{13}})^{\sum_{d|n,d>c_1}\varphi(d)}>|\alpha|^{\frac{11}{13}(n-c_1^2)}>|\alpha|^{\frac{11}{26}n}.
    \end{align*}
    
\end{proof}

Now we consider elements in $X_k$. By Theorem \ref{1} (3), $\tau(p^k)=P_{k+1}(\alpha_p,\beta_p)$ where $\alpha_p,\beta_p$ are the roots of the quadratic $x^2-\tau(p)x+p^{11}$. Since $\alpha_p$ and $\beta_p$ are complex conjugates, they satisfy all the conditions in Corollary \ref{4}. Hence we get:
\begin{equation}\label{Q}
    |\tau(p^k)|>|\alpha_p|^{\frac{11}{26}(k+1)}=(p^{\frac{11}{2}})^{\frac{11}{26}(k+1)}\geq 3^{\frac{121}{52}(k+1)}>2^k
    \textup{  for all }k>2c_1^2-1.
\end{equation}
This shows that the absolute value of $\tau(p^k)$ grows at least exponentially as $k$ grows. Moreover, if $N>2^{2c_1^2-1}$ is a fixed number, then for all $k>\frac{\log N}{2\log2}$, we have $X_{2k}\cap[-N,N]=\emptyset$.

\section{When $k$ is not large}
In this section, assume $N$ is a sufficiently large fixed number, and  $3\leq k<\frac{\log N}{2\log 2}$. We are going to give an upper bound for $\# ( P\cap X_{2k}\cap [-N,N] )$. \\\\
An equation of the form
\[
F(x,y) = D,
\]
where $F(x,y) \in \mathbb{Z}[x,y]$ is homogeneous and $D$ is a nonzero integer,
is known as a \emph{Thue equation}. We consider the well-known generating function \cite[(4.1)]{BCOT} \begin{equation}
    \frac{1}{1-\sqrt{y}t+xt^2}=\sum_{k=0}^\infty F_k(x,y)t^k=1+\sqrt{y}t+(y-x)t^2+\cdots.
\end{equation}
It is well-known that $F_{2k}(x,y)$ are homogeneous polynomials with integer coefficients. The first $F_{2k}(x,y)$ are as follows:
\begin{align*}
F_{2}(x,y)  &= y - x, \\
F_{4}(x,y)  &= y^{2} - 3xy + x^{2}, \\
F_{6}(x,y)  &= y^{3} - 5xy^{2} + 6x^{2}y - x^{3}.
\end{align*}
As observed in \cite{BCOT}, we have the following:
\begin{equation}\label{thue}
    F_{2k}(x,y)=\prod_{j=1}^k\bigg(y-4x\cos^2{\bigg(\frac{\pi j}{2k+1}\bigg)}\bigg).
\end{equation}
\\\\
The reason we want to consider $F_{2k}(x,y)$ is the following:
\begin{equation}
    \tau(p^{2k})=\frac{\alpha_p^{2k+1}-\beta_p^{2k+1}}{\alpha_p-\beta_p}=F_{2k}(x_p,y_p)=\prod_{j=1}^k(y_p-\alpha_{j,k}x_p)
\end{equation}
where $x_p=p^{11}$, $y_p=\tau(p)^2$, $\alpha_{j,k}=4\cos^2{(\frac{\pi j}{2k+1})}=2(1+\cos{(\frac{2\pi j}{2k+1})})\in[0,4]$. Hence $$X_{2k}=\{F_{2k}(x_p,y_p)\colon p>0 \textup{ is prime}\}.$$
Consider the Thue equation $F_{2k}(x,y) = D$. The attainable values of $D$ for $k \geq 3$ form a rather thin set, so the same is true for $\tau(p^{2k})=F_{2k}(x_p,y_p)$. To justify this idea, we need the following result of Bombieri and van der Poorten \cite{Bombieri}:
\begin{prop}\label{5}
    Let $\alpha$ be an element of a number field $K$ of degree $r$ over the rational field $\mathbb{Q}$\footnote{In \cite{Bombieri} it is any number field $k$. We only need the case $k=\Q$.}. If $0<\zeta\leq1$\footnote{In \cite{Bombieri} it is $0<\zeta\leq\zeta_0$ and $\zeta_0=\min(1,\frac{6}{\sqrt{c_0}})$. But then the authors claim that $c_0=28$ is admissible, so we can choose $\zeta_0=\min(1,\frac{6}{\sqrt{28}})=1$.}, then the number of solutions $\beta\in\mathbb{Q}$ of the inequality\begin{equation}
        |\alpha-\beta|<\frac{1}{64h(\beta)^{2+\zeta}}
    \end{equation}
    does not exceed\begin{equation*}
        \frac{2}{\zeta}\log\log(4h(\alpha))+3000\frac{(\log r)^2}{\zeta^5}\log\bigg(\frac{50\log r}{\zeta^2}\bigg).
    \end{equation*}
    Here $h(\alpha)$ is the absolute height of $\alpha$, i.e. if the minimal polynomial of $\alpha$ over $\mathbb{Q}$ is $\sum_{i=0}^r a_ix^i$, then $h(\alpha)=\max_{0\leq i\leq r}|a_i|$.
\end{prop}
Since $\alpha_{j,k}$ is a root of a $F_{2k}(1,y)$, the degree of $\alpha_{j,k}$ over $\mathbb{Q}$ is at most $k$. And since $|\alpha_{j,k}|\leq 4$, we have $h(\alpha_{j,k})\leq 4^k$. Setting $\alpha=\alpha_{j,k}$ and $\zeta=\frac{1}{2}$ in Proposition \ref{5}, we get:
\begin{cor}\label{7}
    The number of solutions $\frac{y}{x}\in\mathbb{Q}$ of the inequality\begin{equation*}
        |\alpha_{j,k}-\frac{y}{x}|<\frac{1}{64h(\frac{y}{x})^\frac{5}{2}}
    \end{equation*}
    does not exceed
    \begin{equation*}
        4\log((k+1)\log 4)+96000(\log k)^2\log(200\log k).
    \end{equation*}
\end{cor}

Let $d_k=\min_{i\neq j}|\alpha_{i,k}-\alpha_{j,k}|$ be the minimum distance between $\alpha_{j,k}$'s for $1\leq j\leq k$. Then \begin{displaymath}
\begin{split}
d_k&=2\min_{i\neq j}\bigg|\cos{\bigg(\frac{2\pi i}{2k+1}\bigg)}-\cos{\bigg(\frac{2\pi j}{2k+1}\bigg)}\bigg|>2\bigg(1-\cos{\bigg(\frac{2\pi}{2k+1}\bigg)}\bigg)\\&=4\sin^2{\bigg(\frac{\pi}{2k+1}\bigg)}>\bigg(\frac{\pi}{2k+1}\bigg)^2 \textup{ for all }k\geq 3.
\end{split}
\end{displaymath}
If $x_p=p^{11}>N^3$: $$|F_{2k}(x_p,y_p)|=\prod_{j=1}^k|y_p-\alpha_{j,k}x_p|=(x_p)^k\prod_{j=1}^k\bigg|\alpha_{j,k}-\frac{y_p}{x_p}\bigg|.$$
Since $\frac{y_p}{x_p}\in\mathbb{Q}$, $\alpha_{j,k}\in\mathbb{R}$, at least $k-1$ of $|\alpha_{j,k}-\frac{y_p}{x_p}|$ are greater than $\frac{1}{2}d_k$. In order to have $|F_{2k}(x_p,y_p)|\leq N$, at least one of $|\alpha_{j,k}-\frac{y_p}{x_p}|$ is less than $$N(x_p)^{-k}\bigg(\frac{1}{2}d_k\bigg)^{-(k-1)}<(x_p)^{\frac{1}{3}-k}\bigg(\frac{2(2k+1)}{\pi}\bigg)^{2(k-1)}$$
For $3\leq k<\frac{\log N}{2\log 2}$, $k\leq O(\log N)$. Thus there exists some constant $c_2>0$ such that if $N>c_2$ and $3\leq k < \frac{\log N}{2\log 2}$, then $$\bigg(\frac{2(2k+1)}{\pi}\bigg)^{2(k-1)}<\frac{1}{64\cdot5^{\frac{5}{2}}}N^{\frac{1}{6}k}<\frac{1}{64\cdot5^{\frac{5}{2}}}(x_p)^{\frac{1}{18}k}$$ 
Therefore if $N>c_2$ and $3\leq k < \frac{\log N}{2\log 2}$, at least one of $|\alpha_{j,k}-\frac{y_p}{x_p}|$ is less than $$(x_p)^{\frac{1}{3}-k}\bigg(\frac{1}{64\cdot5^{\frac{5}{2}}}(x_p)^{\frac{1}{18}k}\bigg)=\frac{1}{64\cdot5^{\frac{5}{2}}}(x_p)^{\frac{1}{3}-\frac{17}{18}k}\leq\frac{1}{64\cdot5^{\frac{5}{2}}}(x_p)^{-\frac{5}{2}}.$$
Since $0\leq \alpha_{j,k}\leq 4$, if $|\alpha_{j,k}-\frac{y_p}{x_p}|<\frac{1}{64\cdot5^{\frac{5}{2}}}(x_p)^{-\frac{5}{2}}$, then $h(\frac{y_p}{x_p})=\max(|x_p|,|y_p|)<5x_p$.  So we have $$\frac{1}{64\cdot5^{\frac{5}{2}}}(x_p)^{-\frac{5}{2}}<\frac{1}{64\cdot5^{\frac{5}{2}}}\bigg(\frac{1}{5}h(\frac{y_p}{x_p})\bigg)^{-\frac{5}{2}}=\frac{1}{64h(\frac{y_p}{x_p})^\frac{5}{2}}.$$
Thus for $N>c_2$ and $3\leq k<\frac{\log N}{2\log 2}$, there exist $1\leq j\leq k$ such that \begin{equation}\label{xp}
    |\alpha_{j,k}-\frac{y_p}{x_p}|<\frac{1}{64h(\frac{y_p}{x_p})^\frac{5}{2}}.
\end{equation} There are $k$ possible choices for $j$. If in addition $F_{2k}(x_p,y_p)$ is an odd prime, then we must have $\gcd(x_p,y_p)=1$, and in this case the number of integer pairs $(x_p,y_p)$ with \eqref{xp} is equal to the number of fractions $\frac{y_p}{x_p}\in\mathbb{Q}$ with \eqref{xp}. Hence for $N>c_2$ and $3\leq k<\frac{\log N}{2\log 2}$, we have 
\begin{align*}
    &  \#\bigg\{p \in P \colon x_p=p^{11}>N^3, F_{2k}(x_p,y_p)\in P\cap[-N,N]\bigg\}\\
    &\leq 
    k\bigg(4\log((k+1)\log 4)+96000(\log k)^2\log(200\log k)\bigg)=O(k(\log k)^2(\log\log k))
\end{align*} 
Since we also have $k\leq O(\log N)$, there exists some constant $c_3>0$ such that $$k\bigg(4\log((k+1)\log 4)+96000(\log k)^2\log(200\log k)\bigg)<N^{\frac{1}{2}}$$ for all $N>c_3$.  Putting $c_4=\max(c_2,c_3)$, we conclude:

\begin{prop}\label{12}
    There is some constant $c_4>0$ such that if $N>c_4$ and $3\leq k \leq \frac{\log N}{2\log 2}$, then 
    $$\# (P\cap \{F_{2k}(x_p,y_p)\colon p>N^{\frac{3}{11}}\}\cap[-N,N]) < N^{\frac{1}{2}}.$$ 
\end{prop}
If $x_p=p^{11}\leq N^3$, then $p\leq N^{\frac{3}{11}}$. Since $x_p$ and $y_p$ are functions of $p$, obviously $$\#(\{F_{2k}(x_p,y_p)\colon 0<p \leq N^{\frac{3}{11}}\}) < N^{\frac{3}{11}}.$$ Hence we conclude: 
\begin{prop}\label{100}
    There exists some constant $c_4$ such that for all $N>c_4$ and $3\leq k<\frac{\log N}{2\log2}$, $$\# (P\cap X_{2k}\cap [-N,N]) < N^{\frac{1}{2}}+N^{\frac{3}{11}} < N^{\frac{9}{10}}.$$
\end{prop}

\section{Proof of Theorem \ref{T}}
Let $b\in\{2,4,6,7,8,9,10,11,12,13,14,15,16,17,18,19,20,21\}$. Put $P_b=\{l\in P\colon l\equiv b\textup{ (mod 23)}\}$.  Our  Then by Corollary \ref{C}, if $l=23a+b\in S\cap P$, we have $l\in X_{2k}$ for some $k\geq 3$. Now let $N=10^{M+1}>\max(2^{2c_1^2-1},c_4)$. By \eqref{Q}, if $k>\frac{\log N}{2\log2}$, then $X_{2k}\cap[-N,N]=\emptyset$. Therefore $$P_b\cap\bigg(\bigcup_{k=1}^\infty X_{2k}\bigg)\cap[-N,N]=P_b\cap\bigg(\bigcup_{3\leq k\leq \frac{\log N}{2\log2}}X_{2k}\bigg)\cap[-N,N].$$Now let $3\leq k\leq \frac{\log N}{2\log 2}$. Since $P_b$ is a subset of $P$, we have $$\# (P_b\cap X_{2k}\cap[-N,N]) < N^{\frac{9}{10}}$$ by Proposition \ref{100}. Hence we have

\begin{equation}\label{AA}
    \#\bigg(P_b\cap\bigg(\bigcup_{3\leq k\leq \frac{\log N}{2\log2}}X_{2k}\bigg)\cap[-N,N]\bigg) < N^{\frac{9}{10}}\cdot\frac{\log N}{2\log2}=\bigg(10^{\frac{9}{10}}\bigg)^{M+1}\cdot\frac{(M+1)\log10}{2\log2}.
\end{equation} 
On the other hand, let $\pi_{P_b}(x) = \# (\{l\in P_b\colon |l|\leq x\})$. By Prime Number Theorem for Arithmetic Progressions, $\pi_{P_b}(x)\sim \frac{x}{11\log x}$ as $x\to \infty$. Hence there is some constant $c_5>0$ such that for any $x>c_5$, $$\frac{9}{10}\cdot\frac{x}{11\log x}<\pi_{P_b}(x)<\frac{11}{10}\cdot\frac{x}{11\log x}.$$
Realizing that $$\#\bigg(([-10^{M+1},-10^M]\cup[10^M,10^{M+1}])\cap P_b\bigg) = \pi_{P_b}(10^{M+1})-\pi_{P_b}(10^M),$$ we conclude that for all $N=10^{M+1}>10c_5$, 
\begin{align*}
    \pi_{P_b}(10^{M+1})-\pi_{P_b}(10^M)&>\frac{9}{10}\cdot\frac{10^{M+1}}{11\log(10^{M+1})}-\frac{11}{10}\cdot\frac{10^M}{11\log(10^M)}\\
    &=\frac{\frac{9}{10}M\cdot10^{M+1}-\frac{11}{10}(M+1)10^M}{11\log10\cdot(M+1)M}\\
    &>\frac{7M\cdot 10^M}{11\log10\cdot(M+1)M}\\
    &=\frac{7\cdot10^M}{11\log10\cdot(M+1)}.
\end{align*} 
For convenience, we let $$A = \bigg(10^{\frac{9}{10}}\bigg)^{M+1}\cdot\frac{(M+1)\log10}{2\log2}$$ and $$B = \frac{7\cdot10^M}{11\log10\cdot(M+1)}.$$
Since $([-10^{M+1},-10^M]\cup[10^M,10^{M+1}])\cap P_b$ is a subset of $P_b\cap[-N,N]$,  $\#(P_b\cap[-N,N])>B$. Hence  $$\#\bigg(P_b\cap[-N,N]\cap\bigg(\bigcup_{k=1}^\infty X_{2k}\bigg)^c\bigg) > B - A.$$  Since $10^{\frac{9}{10}}<10$ and $\frac{7}{11\log10}>\frac{1}{5}$, there is some constant $c_6>0$ such that $B-A>\frac{1}{5}\cdot\frac{10^{M+1}}{M+1}=\frac{1}{5}\cdot\frac{N}{\log N}$ for all $M>c_6$. Hence if $N>\max(2^{2c_1^2-1},c_4,10c_5,10^{c_6+1})$, then the number of primes in $P_b\cap[-N,N]$ which are not values of $\tau(n)$ is greater than $\frac{N}{5\log N}$. Hence $\pi_{P_b\cap T}(N) \gg \frac{N}{\log N}$.\qed  
\begin{cor}
    $\pi_T(N) \gg \frac{N}{\log N}$.
\end{cor}

\section*{Acknowledgments}
The author thanks Professor Michael Larsen for his guidance and encouragement during the preparation of this paper and the anonymous referee for a careful reading of the manuscript and many helpful comments, which have significantly improved the presentation.

\end{document}